\documentclass[12pt]{amsart}
\pdfoutput=1

\usepackage{amssymb,amsmath,amsfonts,eurosym,geometry,ulem,graphicx,caption,xcolor,setspace,comment,footmisc,caption,pdflscape,subfigure,array}
\usepackage[numbers]{natbib}
\usepackage{hyperref}
\usepackage[perc]{overpic}
\usepackage[all]{xy}
\usepackage{todonotes}

\newtheorem{theorem}{Theorem}[section]

\newtheorem{proposition}[theorem]{Proposition}
\newtheorem{lemma}[theorem]{Lemma}

\newtheorem{claim}[theorem]{Claim}

\newtheorem{remark}[theorem]{Remark}

\makeatletter
\def\@setthanks{\vspace{-\baselineskip}\def\thanks##1{\@par##1\@addpunct.}\thankses}
\makeatother

\DeclareMathOperator{\PSL}{PSL_2}
\DeclareMathOperator{\R}{\mathbb{R}}
\DeclareMathOperator{\HH}{\mathbb{H}}

\DeclareMathOperator{\Mod}{S_\text{Mod}}
\DeclareMathOperator{\M}{\overline{S}_{{\text{Mod}}}}
\newcommand{\NN}{\mathcal{N}}

\title{Graph manifolds that admit arbitrarily many Anosov flows}
\author{Adam Clay}
\address{University of Manitoba, Winnipeg}
\email{Adam.Clay@umanitoba.ca}
\thanks{A.C. was partially supported by NSERC grant RGPIN-2020-05343.}
\author{Tali Pinsky}
\address{The Technion, Haifa}
\email{talipi@technion.ac.il}
\thanks{T.P.\ was supported by the Israel Science Foundation (grant No. 51/4051).}
\date{\today}

\begin{document}
\maketitle

\begin{abstract}
For each natural number $n$, we construct an example of a graph manifold supporting at least $n$ different Anosov flows that are not orbit equivalent. 
Our construction is reminiscent of the Thurston-Handel construction \cite{HandelThurston}: we cut a geodesic flow on a surface of constant negative curvature into two pieces, modify the flow in each piece by pulling back to finite covers, and glue together compatible pairs of pullback flows along their boundary tori to get many distinct flows on the resulting graph manifold. 
\end{abstract}

\section{Introduction}
An Anosov flow, also called a hyperbolic flow, is a flow for which some directions are expanded and others are contracted. Such flows are fundamental examples for (idealized) chaotic dynamical systems. 
The study of Anosov flows in dimension three is well-connected to the study of the topology of the three-manifold carrying such flows.
For instance,  a three-manifold carrying an Anosov flow is always irreducible \cite{Palmeira}, has a fundamental group of exponential growth \cite{PlanteThurston}, and carries a tight contact structure \cite{Mitsumatsu}.

In some cases, one can classify the Anosov flows supported by a three-manifold if the manifold has a simple geometry. For instance \cite{ghys2007knots} and \cite{Barbot_Waldhausen} proved that a flow on a Seifert fibered three-manifold is equivalent to a geodesic flow up to finite covers, while \cite{PlanteSolvable} proved that an Anosov flow on a Solv manifold is a suspension of an Anosov diffeomorphism of the torus.

Although any three-manifold carrying an Anosov flow is irreducible, it may be toroidal, i.e. it can contain essential embedded tori.
When this is the case, a fundamental tool of three-dimensional topology allows one to cut the manifold along a collection of the essential tori, obtaining pieces with boundary that have better understood topological and geometric properties.

Therefore, it makes sense to cut the manifold along essential tori and analyse the resulting flows on the geometric pieces.
This is the subject of a number of seminal papers by Barbot and Fenley. In particular, they prove in \cite{barbotfenleytoroidal} that every essential torus is homotopic to one which is either transverse to the flow, or is  \emph{quasi-transverse}: It is transverse except along a finite number of periodic orbits.
The manifold can then be cut along these transverse or quasi-transverse essential tori, from which they have developed a variety of classification results for the Anosov flows on the resulting pieces.

 Essential tori have also been used in the converse direction, namely to glue pieces of Anosov flows together to produce 
new Anosov flows with surprising qualities.
See for instance the examples of Fried-Williams, who construct a non-transitive Anosov flow, and Handel-Thurston who construct an Anosov flows which is transitive, but is neither a geodesic flow nor a suspension.

A central problem in the field has been to determine the number of Anosov flows that can be supported by a single manifold.  For example, this appears as Problem 3.53 of Kirby's problem list, where he asks: Given an integer $N$, does there exist a hyperbolic $3$-manifold with at least $N$ Anosov flows which are topologically inequivalent?

The first explicit examples of manifolds supporting multiple Anosov flows were constructed in \cite{Barbot_Waldhausen}, where Barbot constructs a family of graph manifolds that each support two distinct Anosov flows.  
The surgery techniques of Goodman \cite{Goodman_Surgery} may also produce two distinct Anosov flows on a manifold if the periodic orbit used for the Dehn surgery admits two distinct purely cosmetic surgery slopes.
More recently, Bonatti, Beguin and Yu proved a general theorem allowing  them to glue many pieces with transverse toral boundaries \cite{beguin2014building}. They use this new technique to construct, for each $n >0$, a non-geometric three-manifold $M_n$ consisting of two hyperbolic pieces and one Seifert fibered piece which supports at least $n$ distinct Anosov flows. A similar result was recently obtained by Bowden and Mann using an analysis of the rigidity properties of certain fundamental group actions \cite{mannbowden}, but their manifolds $M_n$ are hyperbolic (thus answering the question from Kirby's problem list). 

A recurrent theme in these examples is that a three manifold must have increasingly ``complicated topology" in order to support an increasing number of distinct Anosov flows.  For example, in \cite{beguin2014building}, this manifests itself as a choice of increasingly complex pseudo-Anosov diffeomorphisms of a surface in order to construct certain hyperbolic pieces in the JSJ decomposition of their manifold $M_n$. 



The main result of this work is to produce new examples of three-manifolds admitting arbitrarily many non-equivalent Anosov flows, with each manifold having particularly simple topology.  For each $n>0$ we construct, via modifications to a geodesic flow, a graph manifold which supports arbitrarily many Anosov flows, no two of which are orbit equivalent. 

The manifolds we construct are therefore topologically simpler than those in the literature, in the sense that the examples of \cite{mannbowden, beguin2014building} both use hyperbolic pieces to achieve the necessary complexity to support arbitrarily many Anosov flows.  
Moreover, we use the simplest possible graph manifolds, each having only two pieces, each piece being the exterior of a trefoil. Indeed, in order to increase the number of Anosov flows supported by the manifolds arising from our construction, we need only increase the shear factor of a certain shear mapping used to identify the boundary tori of the two pieces.

Our technique is inspired by the examples of Handel-Thurston.  They begin by fixing a closed hyperbolic surface $S$, and equipping the unit tangent bundle $T^1S$ with the geodesic flow $\psi$.  Next, they choose a simple closed geodesic $g(t)$ in $S$, and cut $S$ along $g(t)$ to produce two pieces $P_1$ and $P_2$, and equip each of $N_1 = T^1P$ and $N_2 = T^1P_2$ with $\psi|N_i$.  Last, they build a gluing map $F: \partial T^1N_1 \rightarrow \partial T^1N_2$ which is distinct from the identity, but which behaves very will with respect to the flow on each piece, so that the pieces can be assembled into a new manifold $N_1 \cup_F N_2$ equipped with an Anosov flow built from the flows $\psi|N_i$. 
For any of these gluings, the resulting flow is not automatically Anosov. The difficulty is that an Anosov flow carries two invariant foliations, one uniformly (exponentially) attracting and one uniformly repelling.
The gluing usually cannot be made to preserve both these foliations, the smoothness of the flow, and the uniformity of the attraction.
Therefore, one must typically appeal to cone-field arguments to prove that the resulting flow is indeed Anosov.

Our method of proof is as follows: We consider the geodesic flow on the modular surface, which is a flow on the trefoil complement.
 We then observe that the trefoil complement admits an $k$-sheeted self-covering $p_k: M \rightarrow M$ for certain values of $k$.  Fixing a carefully chosen gluing map $F_n$ for $n>0$, and closely following the Handel-Thurston analysis of the invariant directions for the geodesic flow, we are able to find infinitely many pairs of integers $(k, m)$ with corresponding coverings such that the pullback of the geodesic flow from $M$ along $p_k$ and $p_m$ are ``compatible" with the map $F_n$.  With such a choice, the gluing arguments of Handel--Thurston show that the lifted geodesic flows may be glued to produce an Anosov flow on $M \cup_{F_n} M$.  
For distinct pairs of integers $(k,m)$ corresponding to compatible pullbacks, the resulting flows on $M \cup_{F_n} M$ are never orbit equivalent. It follows that:

\begin{theorem}\label{thm:Main}
For every natural number $n$, there exists a graph manifold $M_n$ supporting at least $n$ Anosov flows, no two of which are orbit equivalent.
\end{theorem}

Note that it is impossible to obtain a similar result by fixing a geodesic flow on a \emph{closed} surface and varying the covering maps, as these coverings will have different domains. In fact, in a recent result of Barbot and Fenley \cite{barbotfenley_contact}, they prove that any cover, along the fiber direction, of a unit tangent bundle can carry at most two distinct Anosov flows up to topological equivalence. Thus, it is cardinal here to consider Anosov flows on manifolds with boundary.

\subsection*{Acknowledgments}
We are grateful to Fran\c{c}ois B\'eguin, Christian Bonatti, Sergio Fenley, Kathryn Mann and Bin Yu for making comments on earlier versions of this paper.

\section{background}
\subsection{The geodesic flow}\label{sec:geodesic}
Let $S$ be a hyperbolic surface constructed as $\HH^2/\Gamma$ for some discrete subgroup $\Gamma$ of $\text{Isom}^+(\HH^2)$.
The surface $S$ thus inherits a Riemannian metric from $\HH^2$, which allows one to define the geodesic flow on $S$,
defined on the unit tangent bundle $T^1(S)\cong \PSL(\mathbb{R})/\Gamma$.

A flow $\phi^t$ on a three manifold $M$ is called \emph{Anosov} if there is a continuous $D\phi$-invariant decomposition of the tangent bundle $TM=E^s\oplus E^u\oplus E^t$ and constants $A>0$, $\lambda>1$, so that for all $t \in \mathbb{R}$:
\begin{itemize}
    \item $E^t$ is tangent to $\phi^t$ at any point $x\in M$,
    \item for any $x\in M$ and any $v\in E^u_x$,
    $||D\phi^t(v)_{\phi^t(x)}||\geq A\lambda^t||v_x||,$
    \item for any $x\in M$ and any $v\in E^s_x$,
    $||D\phi^t(v)_{\phi^t(x)}||\leq A\lambda^{-t}||v_x||.$
\end{itemize}
In this decomposition, $E^u$ is called the \emph{strong unstable} direction, and $E^s$ the \emph{strong stable} direction.

Anosov \cite{anosov_geodesic_flows_67} proved that the geodesic flow on $\HH^2$ is Anosov, and that this property descends to any hyperbolic surface $S$.  We can describe the geodesic flow, together with its stable and unstable directions, as follows:

Each point in $T^1\HH^2$ consists of a point $x\in\HH^2$ and a unit vector $v\in T_x\HH^2$. There exists a unique geodesic $g$ in $\HH^2$ passing through $x$ and tangent to $v$. Denote the emanating point of $g$ on $S^1=\partial\HH^2$ by $v^-$, and its terminating point by $v^+$. We denote the point $x$ with direction $v$ by $(x,v)$. The strong stable direction at $(x,v)$ is the direction of the horocycle passing through $x$ based at $v^+$ with a direction perpendicular to the horocycle itself at any point. The strong unstable direction is likewise given by the horocycle based at $v^-$ with a similar perpendicular direction, see Figure~\ref{fig:Horo}.

The geodesic flow moves $(x,v)$ forward along $g$ by a time $t$, while mapping the stable horocycle to another stable horocycle based at $v^+$ (exponentially contracting in $t$ the distance along the horocycle), and the unstable horocycle to an unstable horocycle based at $v^-$.
\begin{figure}[ht!]
\centering
\begin{overpic}[width=5cm]{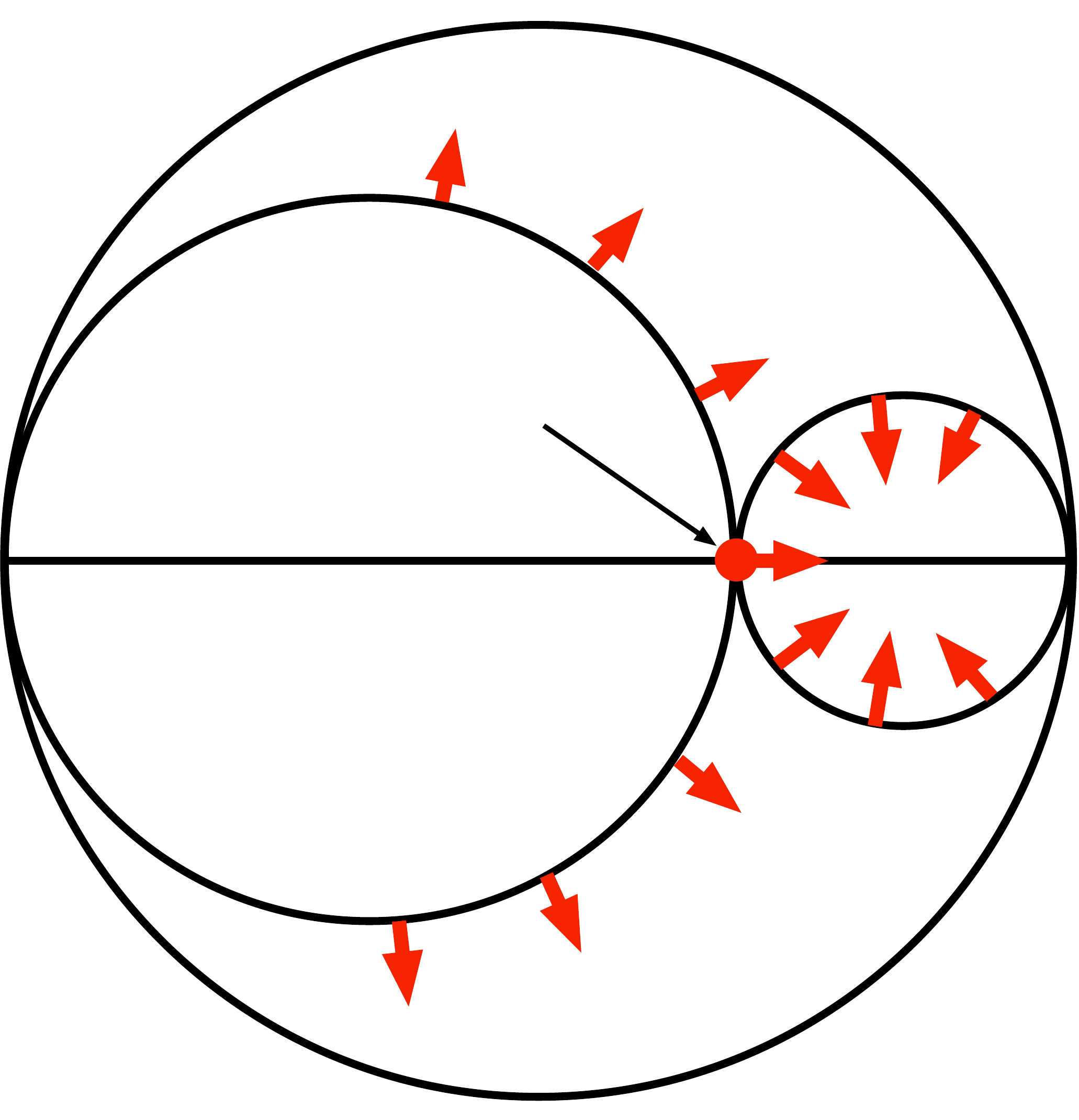}
\put(-10,48){$v^-$}
\put(100,48){$v^+$}
\put(32,63){$(x,v)$}
\put(46,26){${E}^u$}
\put(78,67){${E}^s$}
\end{overpic}
\caption{A point $(x,v)\in T^1\HH^2$ together with the horocycles determining the stable and unstable manifolds at $(x,v)$.}\label{fig:Horo}
\end{figure}
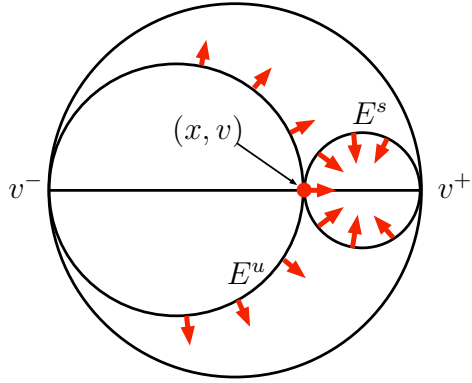 

In order to perform the gluings in Section~\ref{sec:Gluing}, we need to identify the directions of the Anosov decomposition $TM=E^s\oplus E^u\oplus E^t$ for $M=T^1\HH^2$, and later determine these directions upon passing to a quotient $T^1S$ for a particular choice of surface $S$. 
To this end, we analyze the action of the flow on $T(T^1\HH^2)$ by first introducing a coordinate system used in \cite{HandelThurston}.

Fix a point $x \in\HH^2$ and $v\in T^1_x\HH^2$.
Following \cite{HandelThurston}, we identify 
$T_{(x,v)}(T^1\HH^2)$ with $T_x\HH^2\times\R$ as follows:
At any point $y\in\HH^2$ we define the angle $\sphericalangle(w,u)$ between any two vectors $w,u\in T_y^1\HH^2$ to be the angle measured counterclockwise from $w$ to $u$.
Given a vector $w\in T_x\HH^2$ and a real number $\rho\in\R$, define $c_{w,\rho}(t)=(\exp(tw),u(\alpha+t\rho)\,)$ where $\alpha=\sphericalangle(w, v)$ and $u(\alpha+t\rho)\in T^1_{\exp(tw)}\HH^2$ is the vector satisfying $\sphericalangle(T\exp(tw), u(\alpha+t\rho))=\alpha+t\rho$ (see Figure~\ref{fig:Path}). The path $c(t): = c_{w, \rho}(t)$ in $T^1\HH^2$ satisfies $c(0)=(x,v)$, and we identify the point $(w, \rho) \in T_x\HH^2 \times \R$ with $\left.\frac{dc}{dt}\right|_{t=0}$ in $T_{(x,v)}(T^1\HH^2)$. 

\begin{figure}[ht]
\centering
\begin{overpic}[width=8cm, trim=0 0 0 -60, clip]{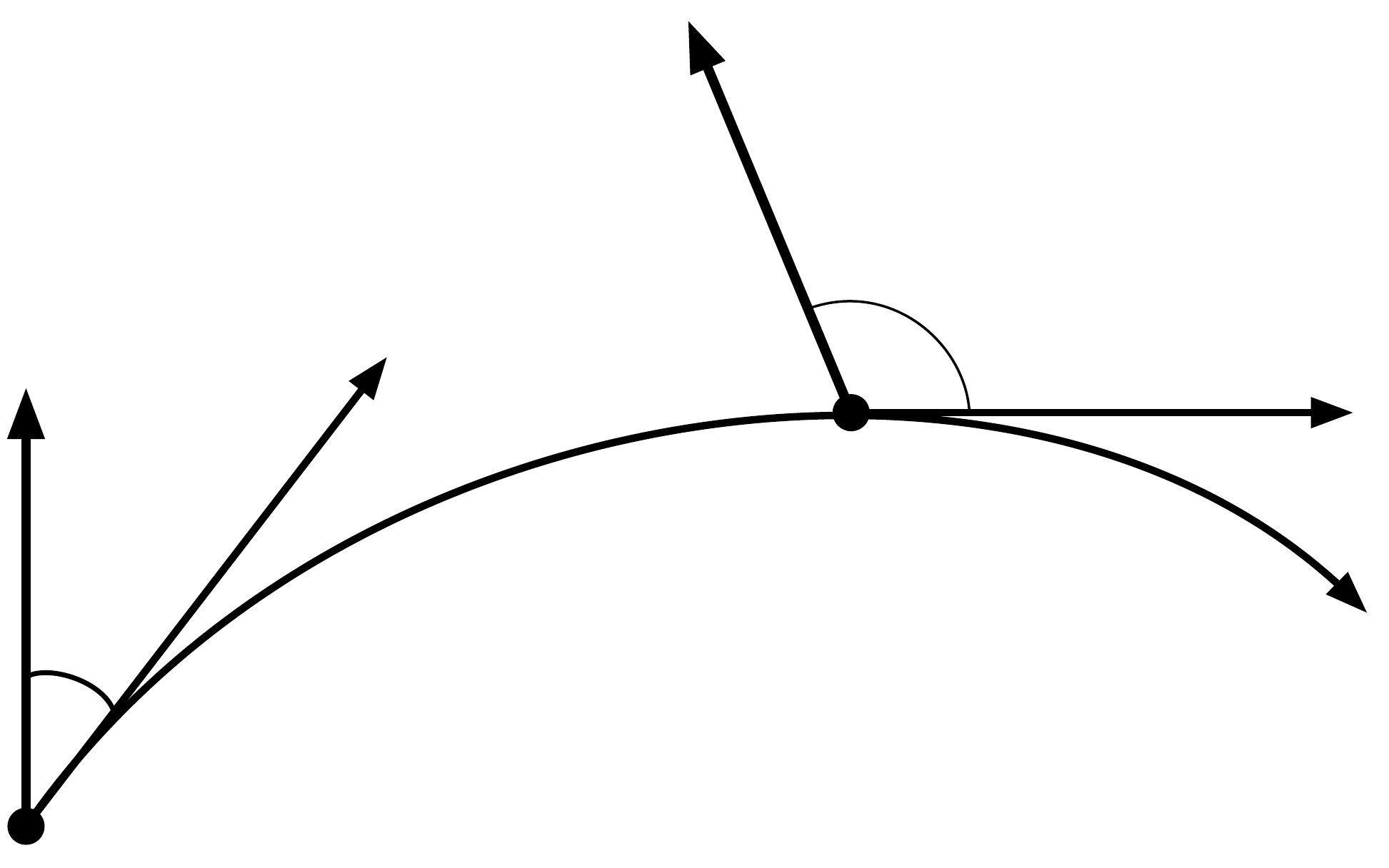}
\put(0,-2){$x$}
\put(90,12){$\exp(wt)$}
\put(27,38){$w$}
\put(0,36){$v$}
\put(38,63){$u(\alpha+t\rho)$}
\put(100,32){$T\exp(tw)$}
\put(5,16){$\alpha$}
\put(70,40){$\alpha+t\rho$}
\end{overpic}
\caption{The path $c(t)$ in $T^1\HH^2$.}\label{fig:Path}
\end{figure}

For any point $(y,w)\in T^1\HH^2$, define the vector $w^\perp$ to be the vector in $T_y^1\HH^2$ satisfying $\sphericalangle(w^\perp,w)=\pi/2$.
Consider again Figure~\ref{fig:Horo} showing the Anosov directions. The strong stable and unstable manifolds are both tangent to $v^\perp$ (that is, both are pointing downwards from $x$ in the figure). For the unstable manifold (that is, the horocycle on the left), the direction vector at a point rotates clockwise as one proceeds along the horocycle, while for the stable manifold the rotation is counterclockwise.  Thus in our coordinate system:
\begin{itemize}
    \item $E_{(x,v)}^t$ is generated by $(v,0)$,
    \item $E_{(x,v)}^u$ is generated by $(v^\perp,-1)$,
    \item $E_{(x,v)}^s$ is generated by $(v^\perp,1)$.
\end{itemize}

In particular, the directions in the $(v^\perp,\rho)$ plane corresponding to the unstable manifold lie in the second and fourth quadrants.  We will need this information for the gluing arguments in Section \ref{sec:Gluing}. 

\begin{remark}\label{rem:Trivializing}  Fixing a geodesic $g$ in $\HH^2$, we can consider the unit tangent bundle to $g$. At any point $(x,v)\in T^1g$,
parallel transport along $g$ preserves the angle $\alpha=\sphericalangle(Tg,v)$. One may identify $T^1g$ with $\R\times\R/2\pi$ by $(t,\theta)\mapsto(g(t),v)$ where $v$ is the vector in $T^1_{g(t)}\HH^2$ satisfying $\sphericalangle(Tg,v)=\theta$. In these coordinates the angles of the Anosov directions above do not depend on the point t, but only on the angle $\alpha$.
\end{remark}

\subsection{The modular surface}\label{sec:modular}
Let $\Gamma$ denote the group $\PSL(\mathbb{Z})$, with generating set
$$\left\{ \begin{pmatrix}
0 &1\\
-1& 0
\end{pmatrix},
\begin{pmatrix}
0 &-1\\
1& 1
\end{pmatrix} \right\}.$$  
In the  standard representation of  $\text{Isom}^+(\HH^2)$ as $\PSL(\R)$ acting by M\"obius transformations, these generators of $\Gamma \subset \PSL(\R)$  act respectively as rotation by $\pi$ about the point $p=i$, and rotation by $2\pi/3$ about the point $q=\frac{1}{2}+\frac{\sqrt{3}}{2}i$.
The modular surface $\Mod$ is the hyperbolic orbifold $\Mod=\HH^2/\Gamma$, a surface with one cusp at infinity and two cone points, one of order 2 and one of order 3.

Increasing the distance between the centers of rotations $p$ and $q$ in $\HH^2$ as in \cite{ghys2007knots} turns the cusp into a funnel. There is then a unique closed geodesic around the funnel, that we denote by $g$. The greater the distance between $p$ and $q$, the longer the geodesic $g$. We cut the surface along $g$ to obtain a compactification of $\Mod$ which is a surface with a boundary. The resulting surface $\M$ is depicted in Figure~\ref{fig:Modular}.

\begin{figure}[ht]
\centering
\begin{overpic}[width=4cm]{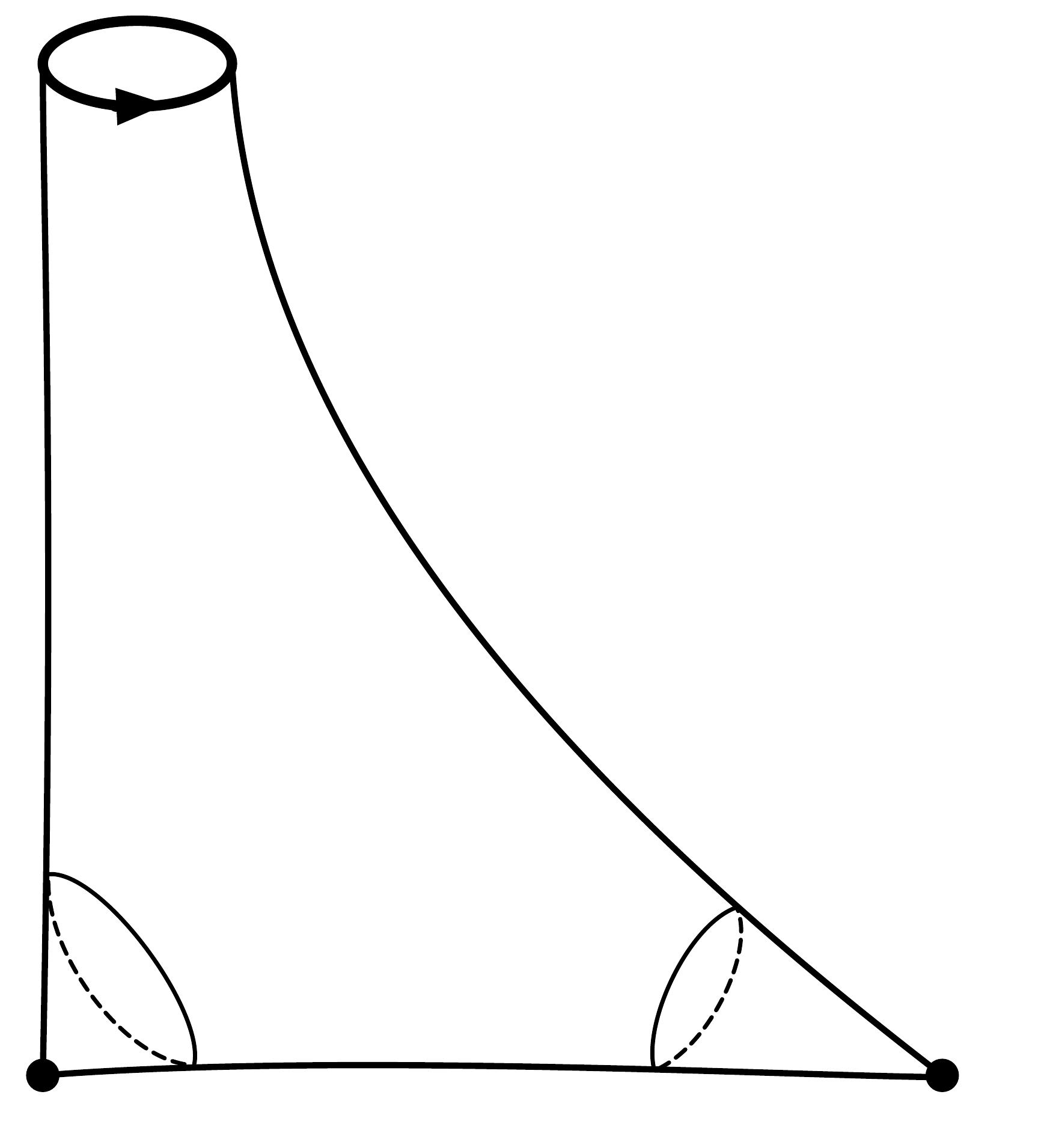}
\put(-2,0){$2$}
\put(86,0){$3$}
\put(10,80){$g$}
\end{overpic}
\caption{The compactified modular surface, $\M$.}\label{fig:Modular}
\end{figure}

The unit tangent bundle $T^1\M$ is the bundle of all $(x,v)$ where $x\in \M$ and $v\in T_x\M$ satisfies $||v||=1$.
Consider the boundary torus $\partial T^1\M=T^1g$. 
We may trivialize the fiber direction along this torus using the tangent to the geodesic $g$ at each point as a section, as in Remark~\ref{rem:Trivializing}. If $g$ has length $L$, this results in the coordinates $[0,L]\times[0,2\pi]$ on the boundary torus, $(s,\theta)\mapsto(g(s),v)$, where $g$ is parametrized by arc-length, and $v\in T^1_{g(s)} \M$ satisfies $\sphericalangle(Tg,v)=\theta$.
We call these the orbit-fiber coordinates. As the geodesic flow is structurally stable, it follows that it is independent of the length $L$ of the boundary geodesic. Hence, we may choose the length to be equal to $2\pi$ so that $0\leq s\leq 2\pi$.

It is uncommon to define Anosov flows on a manifolds with boundary, so for ease of exposition we avoid introducing such a definition and instead introduce our objects of study as pieces of an Anosov flow on a closed $3$-manifold cut along essential tori. 

Consider the surface in Figure \ref{fig:Double}, which is a sphere with four cone points, two of order $2$ and two of order $3$, denoted by $S_{2,2,3,3}$. It is the double of the modular surface $\M$ and is also a hyperbolic surface. As in Section~\ref{sec:geodesic}, there is a geodesic flow over $S_{2,2,3,3}$, which is an Anosov flow defined on $T^1 S_{2,2,3,3}$. Cutting the three manifold $T^1 S_{2,2,3,3}$ along the two dimensional torus $T^1g$ as in Figure \ref{fig:Double} one obtains two copies of $T^1\M$, and the geodesic flow on $\M$ is thus a piece of an Anosov flow on a closed manifold as in \cite{barbot_fenley_free}.

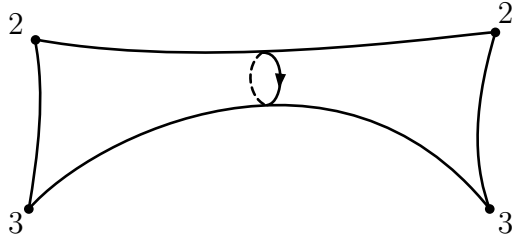
\begin{figure}[ht]
\centering
\begin{overpic}[width=7cm]{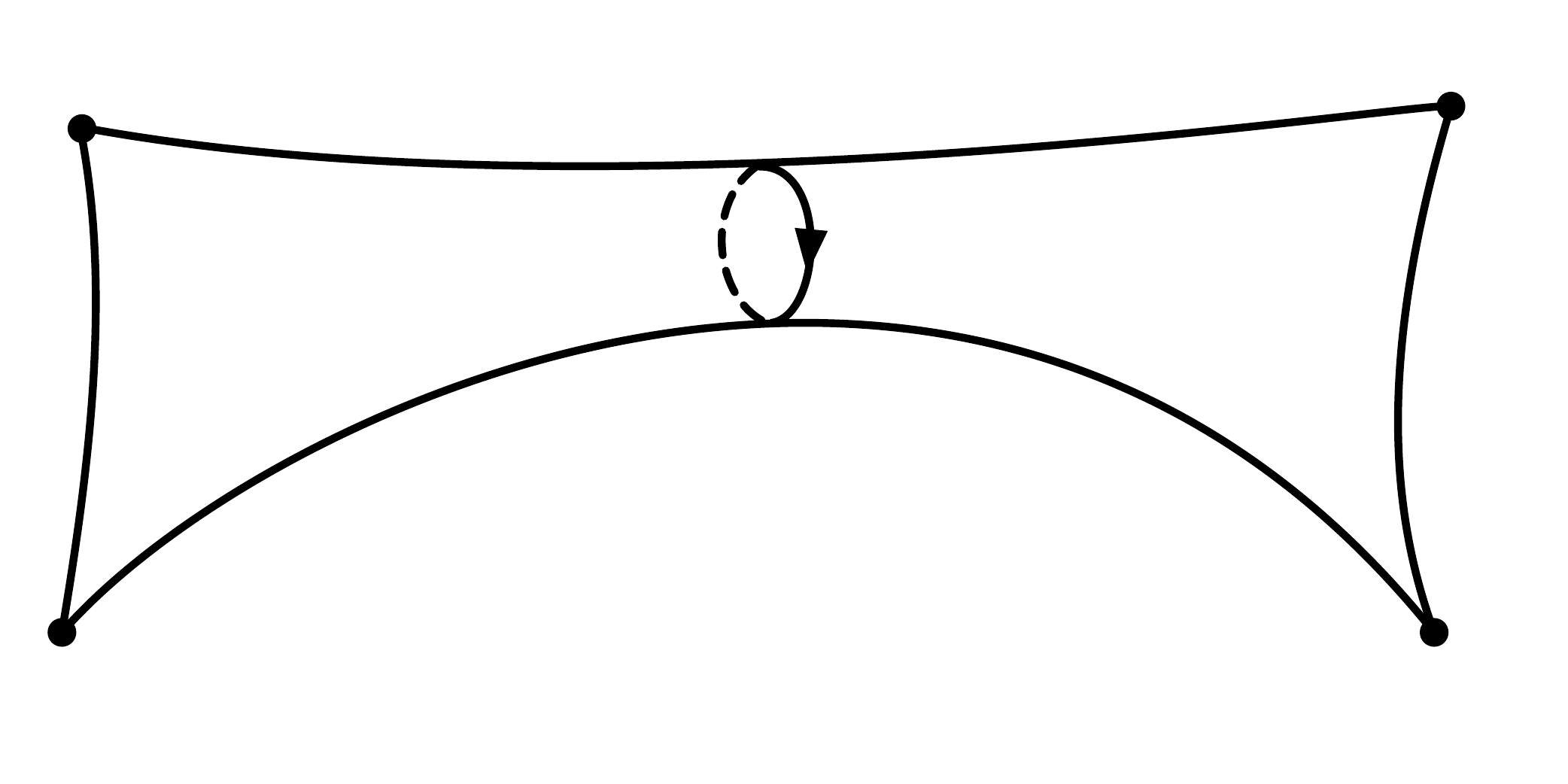}
\put(0,5){$3$}
\put(0,43){$2$}
\put(93,5){$3$}
\put(93,45){$2$}
\end{overpic}
\caption{The double $S_{2,2,3,3}$ of $\M$ with the curve $g$.}\label{fig:Double}
\end{figure}

We may use the coordinate system defined in Section~\ref{sec:geodesic} for $T(T^1\M)$ to identify the Anosov directions over the boundary torus $\partial T^1\M$. 
In our coordinate system $(s,\theta)$ corresponds to the point $(g(s),v)$ where $v$ is the rotation $R_\theta$ by $\theta$ of $Tg$ along the fiber at the point $g(s)$. In particular,
The flow direction at $(g(s),v)$ is generated by $(v,0)=(R_\theta(Tg),0)$.
The strong unstable direction $E_{(g(s),v)}^u$ is generated by $(v^\perp,-1)$, where $v^\perp$ is the direction in the plane perpendicular to $v$, and the unstable direction
 $E_{(x,v)}^s$ is generated by the same vector and slope $+1$,  $(v^\perp,1)$.
Thus, the Anosov directions are all rotations of the directions along $g$, all by the same angle $\theta$ in a horizontal (tangent to the surface) direction.
In particular, the Anosov directions depend solely on $\theta$ and not on $s$. 
These directions are the same for any hyperbolic surface $\HH^2/\Gamma$, and $g$ a closed geodesic in $S$.

\subsection{The trefoil complement}
In this section we will show how to identify $T^1\M$ with the trefoil complement, and will identify the closed orbit on the boundary of $T^1\M$ with a curve on the boundary of the trefoil complement expressed in meridan/longitude coordinates.

 The trefoil complement and its Seifert fibration can be described as follows.  There is an action of $S^1$ on $S^3$ given by $\lambda \cdot (z_1, z_2)$ = $(z_1 \lambda^2, z_2 \lambda^3)$ for all $\lambda \in S^1$.  Each orbit of this action is a trefoil knot lying on one of the tori $\{(z_1, z_2) \mid |z_1|^2=r\}_{0<r<1}$, except for the two orbits $$S^1 \times \{0\}=\{(z_1,z_2)\mid |z_1|=1, z_2=0\}  
\mbox{ and } \{0\} \times S^1=\{(z_1,z_2)\mid z_1=0, |z_2|=1\}.$$
\begin{figure}[ht]
\centering
\begin{overpic}[width=8cm]{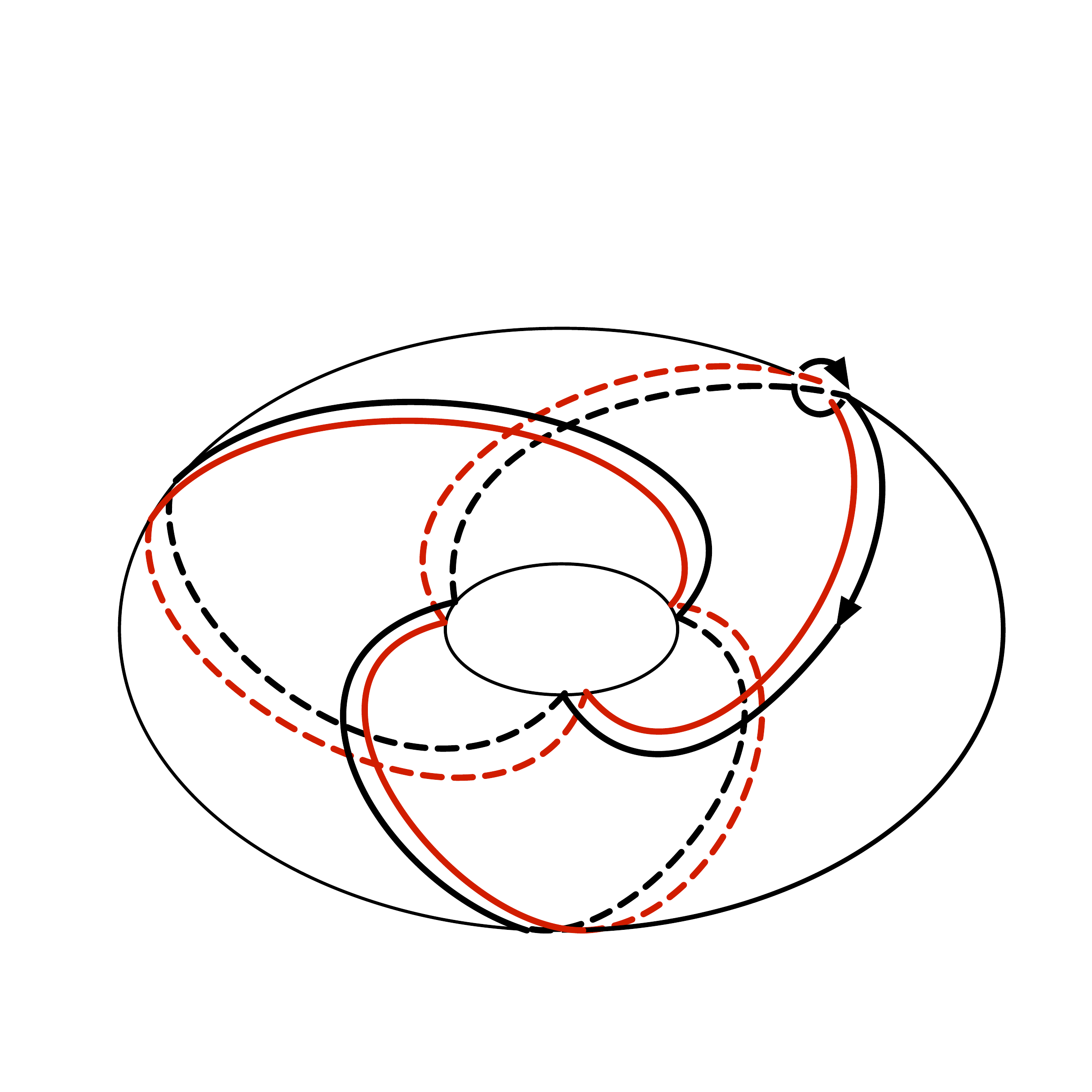}
\put(78,68){$\mu$}
\put(79,40){$h$}
\end{overpic}
\caption{The complement of the trefoil knot.}\label{fig:Trefoil}
\end{figure}
Fixing a choice of orbit $T_{2,3} \subset T=\{ (z_1, z_2) \mid |z_1|^2 = |z_2|^2\}$, we may remove an open neighbourhood of $T_{2,3}$ from $S^3$ consisting of a union of fibers (orbits) to get  $M = S^3\setminus \NN (T_{2,3})$, a compact manifold with boundary, as in Figure \ref{fig:Trefoil}.  The decomposition of $M$ into orbits under the $S^1$ action gives $M$ the structure of a Seifert fiber space, where the orbit $S^1 \times \{0\}$ admits a fibered torus neighbourhood with invariants $(2,1)$ and $\{0\} \times S^1$ admits a fibered torus neighbourhood with invariants $(3,1)$.  The orbit surface for this fibration, i.e. the space obtained by collapsing each fiber to a point, is a copy of $S^2$ minus a disk, with two cone points of orders $2$ and $3$, i.e. it is the surface $\M$.

The trefoil complement also has a fibration over $S^1$: Consider the punctured torus $\text{T}_0$ composed of two disks in the solid torus $\{(z_1, z_2) \mid |z_1|^2 \leq |z_2|^2 \}$ connected by three half twisted bands in the solid torus $\{(z_1, z_2) \mid |z_1|^2 \geq |z_2|^2 \}$ (see \cite{Rolfsenknots} for more details, including an explicit parameterization of this surface).  On the boundary torus $\partial M$, we can choose a basis composed of the meridian $\mu$ of the torus $\NN (T_{2,3})$, and a regular fiber $h$ of the Seifert fibration (see Figure~\ref{fig:Trefoil}). The longitude $\lambda$ is the boundary of the Seifert surface  $\partial \text{T}_0\subset \partial M$. 

The Seifert fibers intersect the surface $\text{T}_0$ transversely, with each regular fiber intersecting it six times.  We assume that these intersections happen at regular intervals along the fiber (that is, if $\lambda, \lambda' \in S^1$ are such that $\lambda \cdot (z_1, z_2) = \lambda'\cdot (z_1', z_2')$ for $(z_1, z_2), (z_1', z_2') \in \text{T}_0$, then $\lambda^{-1} \lambda' = e^{2k \pi i / 6}$ 
for some $k \in \{0, \ldots, 5\} $).  As such, every point in $M$ can be written uniquely as $\lambda \cdot (z_1, z_2)$ for $(z_2, z_2) \in \text{T}_0$ and $\lambda \in S^1$ with $0 \leq \arg(\lambda) < \pi/3$. The singular fibers intersect the surface $2$ and $3$ times respectively.
Sliding $\text{T}_0$ along the fibers rotates $\text{T}_0$ through $S^3$ until it returns to its initial position as a set---in other words, $e^{2\pi i/6} \cdot \text{T}_0 = \text{T}_0$. Note this will shift each intersection point of a regular Seifert fiber with $\text{T}_0$ to the next intersection point along the fiber, permuting cyclically the two disks on the outside, and permuting the three bands on the inside. This yields a map $\varphi : \text{T}_0 \rightarrow \text{T}_0$ given by $\varphi(x) = e^{2\pi i/6} \cdot x$ for all $x \in \text{T}_0$.

As each Seifert fiber has $2$, $3$ or $6$ intersection points with the punctured torus, $\varphi^6=\text{id}$.  Moreover, $\varphi$ serves as the generator of the group of deck transformations for the natural branched covering map $p : \text{T}_0 \rightarrow \M$ as in Figure~\ref{fig:Cover}.

\begin{figure}[ht]
\centering
\begin{overpic}[width=9cm]{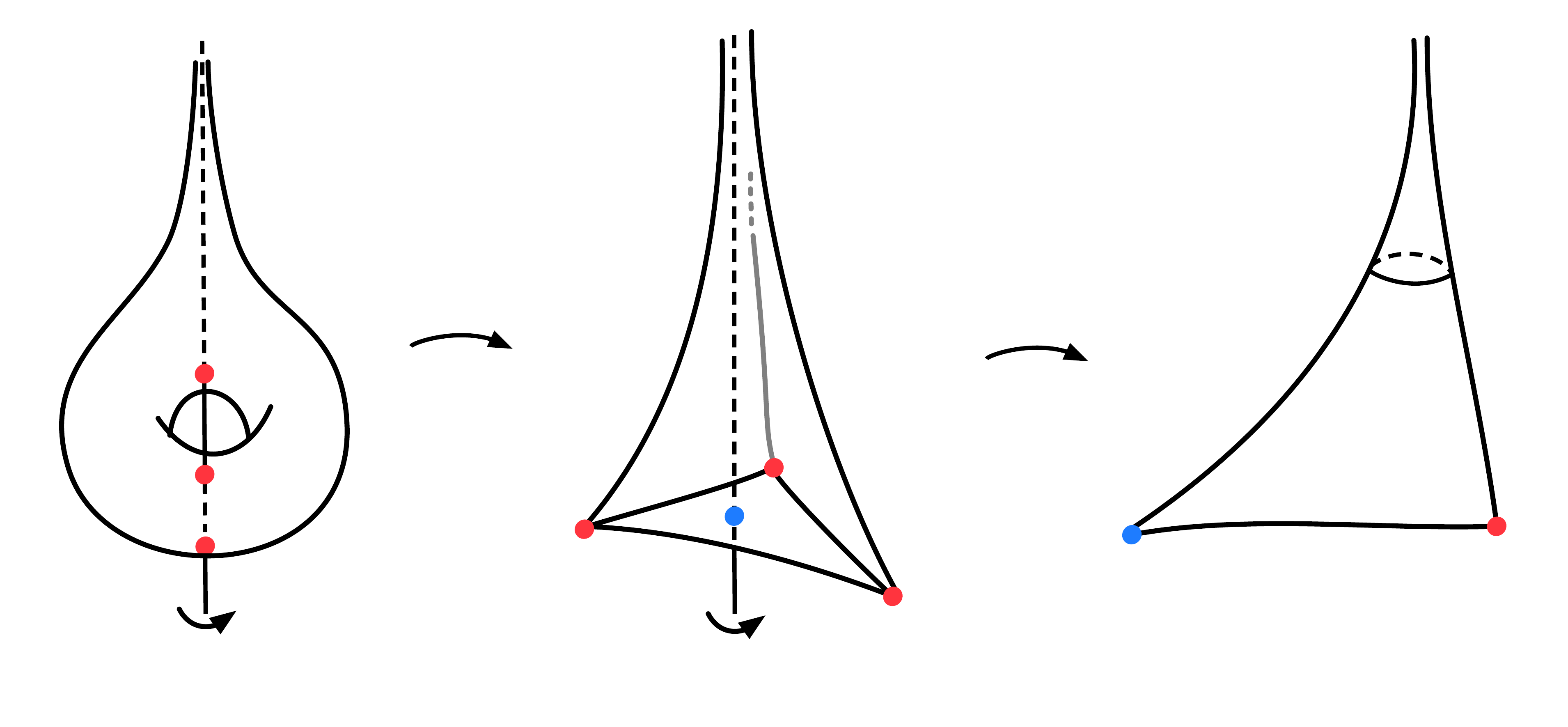}
\put(10,0){$\pi$}
\put(43,0){$2\pi/3$}
\end{overpic}
\caption{The cover $p:\text{T}_0 \rightarrow \Mod$, here $\text{T}_0$ is identified with the surface $T^2 \setminus D^2$ for ease of illustration.}\label{fig:Cover}
\end{figure}

The unit tangent bundle to $T^1\M$ is composed of two unit tangent bundles, each one a unit tangent bundle to a closed neighbourhood of a cone point, glued along the unit tangent bundle to a segment. It is thus homeomorphic to the (compactified) trefoil complement $M$ \cite{milnor1968singularities}.

\begin{proposition}[Ghys]
\label{orbit_slope}
The identification of $T^1 \M$ with the trefoil complement $M$ carries the closed orbit of the geodesic flow on $\partial T^1 \M$ to the curve $\mu \subset \partial M$.
\end{proposition}

\begin{proof}
The proof is given in \cite{ghys2007knots}, and we include it here for completeness. The covering map $p$ induces a cover $P :T^1 \text{T}_0  \rightarrow M\cong T^1(\M)$ and $\varphi$ induces $\Phi: T^1 \text{T}_0 \rightarrow T^1F$. 
As the unit tangent bundle $T^1 \text{T}_0$ is trivial, it has a section, say $s$.  Set $\bar{\lambda} = s( \partial \text{T}_0)$, then we can assume $\bar{\lambda}$ appears
as in Figure \ref{fig:Longitude}, where the black arrows indicate the direction of the nonsingular vector field defined by $s$. The images $\Phi^k \circ s$ are disjoint for $k \in \{0, \ldots, 5\} $. Thus, the image under $P$ of any such section is a copy of $\text{T}_0$ embedded in the trefoil complement  $M$ with $\partial \text{T}_0 \subset \partial M$. Hence, we may identify the image of $P \circ s$ with the Seifert surface for the trefoil, and in particular, choosing the orientation for $\lambda$ accordingly, $P(\bar{\lambda})=\lambda$. 

We can also identify a curve $\bar{h}$ in $\partial(T^1 \text{T}_0)$ that covers a regular fiber $h$ of the Seifert fibration of $M$.  Fixing $x \in \partial \text{T}_0$, set $\bar{h} = \{(x, v) \mid ||v||=1 \}$, oriented so that counterclockwise rotation of the vector $v$ is the positive direction along $\bar{h}$.  Choosing an appropriate orientation for $h$, we get $P(\bar{h}) = h$.

Next, we define a curve $\beta$ representing the class $[\bar{h}] + [\bar{\lambda}] \in H^1(\partial T^1 \text{T}_0)$ by first supposing that $\partial \text{T}_0$ is identified with a curve $\alpha:[0,L] \rightarrow \text{T}_0$, and choosing a point $t_0 \in [0,L]$ for which the section is a direction tangent to the curve, i.e.  $s(\alpha(t_0)) = (\alpha(t_0), \frac{T\alpha(t_0)}{||T\alpha(t_0)||})$.  
Supposing the longitude is traveled counterclockwise as well, as in Figure \ref{fig:Longitude}, 
such a point corresponds to the rightmost point of the curve $\bar\lambda$ in Figure \ref{fig:Longitude}, and reparameterizing if necessary, we assume $t_0= 0$.  Our curve $\beta: [0,L] \rightarrow T^1 \text{T}_0$ is $\alpha$ with its tangent direction, $\beta(t) = (\alpha(t), \frac{T\alpha(t)}{||T\alpha(t)||})$.  To see that $[\beta] =[\bar{h}]+ [\bar{\lambda}]$, note that if $\bar{\lambda}(t) = s \circ \alpha(t)= (\alpha(t), v_t)$ for some unit vector $v_t \in T_{\alpha(t)} \text{T}_0$, then 
$\sphericalangle(v_t, T \alpha(t)) = \frac{2 \pi t}{L}$.  Thus the vector $v_t$ makes one complete rotation relative to $T \alpha(t)$ as $t$ ranges over $[0,L]$, so that $\beta(t) \sim \bar{\lambda} \circ \bar{h}$.

\begin{figure}[ht]
\centering
\begin{overpic}[width=6cm]{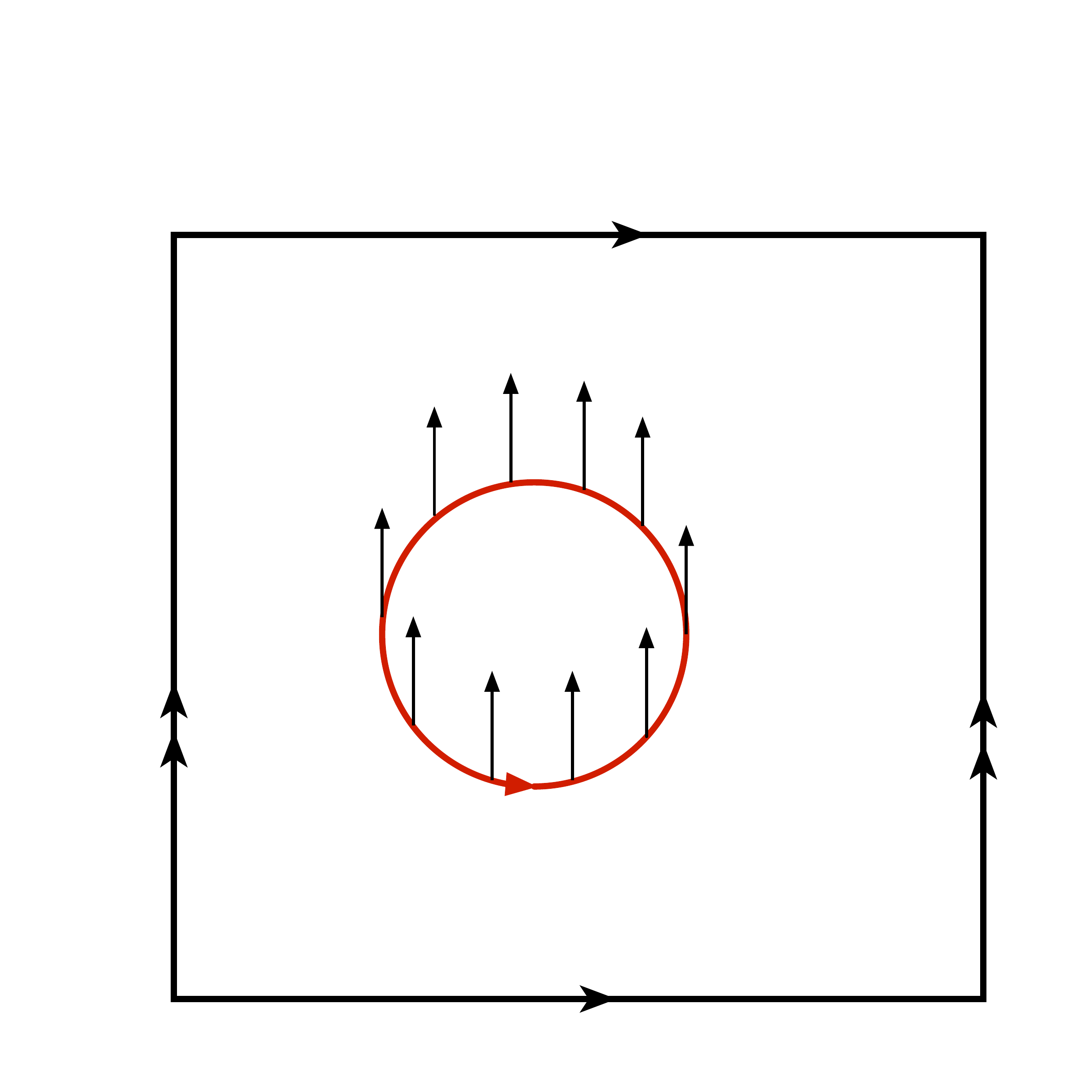}
\put(43,16){$\bar\lambda$}
\end{overpic}
\caption{The curve $\bar\lambda$ on $\partial T^1 \text{T}_0$ in red, as the curve $\partial \text{T}_0$ embedded into $T^1 \text{T}_0$ by assigning a direction to each point.}\label{fig:Longitude}
\end{figure}

Considering Figure \ref{fig:Cover}, we see that the cover $p:\text{T}_0 \rightarrow \Mod$ wraps $\partial \text{T}_0$ six times around the boundary curve $g$ of $\M$, and thus $P(\beta) = (g, \frac{Tg}{||Tg||})$, where $\beta$ wraps six times around its image under $P$.  On the other hand, from Figure~\ref{fig:Trefoil} it follows that the fiber $h$ has linking number $6$ with the trefoil, hence $[h]=\pm[\lambda]\pm6[\mu]$.
However our choices of orientations for $\lambda$ and $h$ yield $[\beta] = [\bar{\lambda}] + [\bar{h}]$, so $[P(\beta)] = [\lambda] + [h] = \pm 6 [\mu]$.

Therefore, the closed orbit $(g, \frac{Tg}{||Tg||})\subset \partial T^1\M$ is exactly the curve you Dehn fill to get $S^3$ (c.f. \cite{ghys2007knots}).
\end{proof}

\vskip .5cm
\section{Lifting}
\label{sec:Covers}

With $M$ as in the previous section, recall that the $S^1$ action which determined the Seifert fibration allowed every point in $M$ to be written uniquely as $\lambda\cdot (z_1, z_2)$ for $(z_2, z_2) \in \text{T}_0$ and $\lambda \in S^1$ with $0 \leq \arg(\lambda) < \pi/3$.  Then for each $d = \pm 1 + 6k$ with $k \in \mathbb{N}_{>0}$ we can construct a $d$-fold covering map $p_d : M \rightarrow M$ as $p_d(\lambda \cdot (z_1, z_2)) = (z_1 \lambda^{2d}, z_2 \lambda^{3d})$ if $d = 1+6k$ and $p_d(\lambda \cdot (z_1, z_2)) = (z_1 \lambda^{-2d}, z_2 \lambda^{-3d})$ if $d = -1+6k$.  From this description, it is clear that under the cover $p_d$ each regular fiber of the Seifert fibration upstairs wraps $d$ times around each regular fiber downstairs (either preserving or reversing orientation depending on the degree of the cover).

This same covering map can also be described in terms of mapping cylinders.  Recall that $M \cong (\text{T}_0 \times [0,1])/ \sim$ , where $(x, 0) \sim (\varphi(x), 1)$ for all $x \in \text{T}_0$ and $\varphi$ is of order $6$.  When $d =1 + 6k$, consider the $d$-fold cyclic cover constructed from $d$ ``puzzle pieces", each homeomorphic to $M$ cut open along a Seifert surface as in \cite{Rolfsenknots}. The resulting manifold is homeomorphic to the mapping cylinder with respect to 
$\varphi^d$. As $\varphi^6=\text{id}$, $\varphi^d = \varphi^{1}$ and the cover  manifold is again homeomorphic to $M$. It wraps, via the covering map $p_d$, $d$ times around itself.  It follows from this description that the preimage $p_d^{-1}(\lambda) = p_d^{-1}(\partial \text{T}_0)$  is $d$ disjoint copies of $\lambda$.

\vskip.5cm

These descriptions and the observations about the action of $p_d$ on the fiber $h$ and the longitude $\lambda$ are sufficient to completely describe the behaviour of the maps $p_d$ upon restriction to the boundary torus $\partial M$.

\begin{lemma}
\label{peripheral map}
Let $M$ be the complement of the trefoil knot in $S^3$, and $p_d: M \rightarrow M$ denote a covering map of order
 $d = 1 + 6k$,
 where $k \in \mathbb{N}$. 
 Then the preimage of the curve $\mu$ under the map $p_d$ is a single curve, of homology  $[\mu] + k [\lambda]$.
 \end{lemma}

\begin{proof}
From the above descriptions of the covering map $p_d$, the induced map $p_d^{*}:\pi_1(\partial M) \rightarrow \pi_1(\partial M)$
 satisfies
$(p_d^{*})^{-1}([h]) = [h]$ and $(p_d^{*})^{-1}([\lambda]) = d [\lambda]$.
Thus, as $[h] = 6[\mu]  - [\lambda]$:
\[ 
(p_d^{*})^{-1}(6[\mu]) = (p_d^{*})^{-1}([\lambda] ) + (p_d^{*})^{-1}([h])   = d [\lambda] +[h] = 6[\mu]+(d-1)[\lambda]
\]
so that  $(p_d^{*})^{-1}(6[\mu])=6[\mu]+6k[\lambda]$ and
\[
(p_d^{*})^{-1}([\mu])=[\mu]+k[\lambda]. \]

What remains to be shown is that $p_d^{-1}(\mu)$ consists of  single curve.
We know that $p_d^{-1}(\lambda)$ consists of $d$ disjoint (parallel) curves, and we can assume that $\mu \cap \lambda$ is a single point. 
It follows that $|p_d^{-1}(\lambda\cap\mu)|=d$ and $\mu$ must be covered by a single curve intersecting each of the $d$ components of $p_d^{-1}(\lambda)$ exactly once.
\end{proof}

\begin{proposition}\label{pro:TwoOrbits}
Suppose $d = 1  + 6k$ and let $\psi_d$ denote the pullback of the geodesic flow along the covering map $p_{d}: M \rightarrow M$.  Then $\psi_d$ has exactly two periodic orbits of slope $[\mu] + k [\lambda]$.  Moreover, if $m$ and $n$ are positive integers with $m, n \equiv 1 \mod 6$, then $\psi_n$ and $\psi_m$ are orbit equivalent if and only if $n=m$.
\end{proposition}

\begin{proof}
That $\psi_d$ has exactly two periodic orbits of the given slope follows from Lemma \ref{peripheral map}, and Proposition \ref{orbit_slope}.

Now suppose that $H :M \rightarrow M$ is a diffeomorphism carrying the orbits of $\psi_m$ to the orbits of $\psi_n$. As the symmetry group of the pair $(S^3, T_{p,q})$ is $\mathbb{Z}_2$, if $H$ is not isotopic to the identity then it is isotopic to a map that generates the symmetry group.  Such a map induces, up to conjugation, the homomorphism $\phi : \langle a, b \mid a^3 =b^2\rangle \rightarrow \langle a, b \mid a^3 =b^2\rangle$ given by $\phi(a) = a^{-1}, \phi(b) = b^{-1}$. 

The homomorphism $\phi$ is conjugate to a homomorphism whose restriction to the peripheral subgroup $\pi_1(\partial M)$ has action $[\mu] \mapsto -[\mu]$ and $[\lambda] \mapsto -[\lambda]$.  As such, a closed orbit $\mu + k \lambda$ will be carried by $H$ to a closed orbit isotopic to $-\mu - k \lambda$, meaning that the closed orbits of $\psi_n$ and $\psi_m$ are never identified by $H$ unless $m = n$.
 \end{proof}

\begin{remark}\label{rem:fiber}  
It follows that there is a choice of longitude $\lambda$ that intersects the orbit at a single point for any cover of degree $d=1+6k$ where $k\in\mathbb{N}_{>0}$.  Thus we can always choose the longitude $[\lambda]$ and a periodic orbit $g_d$ such that $[g_d]$ is a basis for $H_1(\partial M)$. In this basis we have, by the equation in the proof of Lemma \ref{peripheral map}, $
[h]=-d[\lambda]+6[g_d].
$
 \end{remark}

\section{Gluing pieces}\label{sec:Gluing} 
In this section we complete the proof of Theorem \ref{thm:Main}.  For each $n\in\mathbb{N}$ we begin by defining the map $F_n:\partial M \rightarrow \partial M$ as a diffeomorphism inducing the following map on the first homology of $\partial M$:
\[ F_n^*([\mu]) = [\mu]+2n[\lambda], \mbox{ and } F_n([\lambda]) = -[\lambda].
\]

\begin{theorem}\label{thm:Glue}
The manifold $M_n=M\cup_{F_{n}} M$ 
supports $n+1$ Anosov flows, no two of which are orbit equivalent.
\end{theorem}

\begin{proof}
For the geodesic flow $\psi$, we may switch from the orbit-fiber coordinates $(t,\theta)$ to orbit-longitude coordinates $(t,\tau)$ where $0\leq\tau\leq2\pi$ is an arc-length parametrisation of the longitude $\lambda$, normalized to have length $2\pi$. 
Recall that the Anosov directions are independent of $t$, and they are a horizontal rotation by $\theta$ of the directions at $\theta=0$. 
Thus, they are also independent of $t$ in the $(t,\tau)$ coordinates, and consist of a rotation by $\tau$ of the directions at points with $\tau=0$.

Next consider the flow $\psi_d$ corresponding to a self-cover of the trefoil complement of degree $d=1+ 6k$. By the previous section, a tangent orbit $g_d$ is a curve of homology $[\mu] + k[\lambda]$ so that $\{[g_d], [\lambda]\}$ is a basis for the homology of the boundary torus.
The curve $g_d$ is $d$ times the length of the boundary geodesic $g_1$ of $\M$. 
Since the geodesic flow on $M$ is independent of the length of $g_1$ (see Section~\ref{sec:modular}), we may choose its length to be $L=2\pi/d$, so that an arc-length parameter for $g_d$ ranges from 0 to $2\pi$ for any $d$.

 Fix $n > 0$.  For each $k$ with $0 \leq k \leq 2n$, the map $F_n$ is chosen so that
 \[ F_n^*([g_d]) = F_n^*([\mu] + k [\lambda]) = [\mu] + (2n-k) [\lambda] = [g_{d'}]
 \]
 where $d' = 1+6(2n-k)$. In other words, $F_n$ carries any periodic orbit $g_d$ of the flow $\psi_{1+ 6k}$ to a curve isotopic to a periodic orbit $g_{d'}$ of the flow $\psi_{1+ 6(2n-k)}$, with the same orientation. 

At the same time, the periodic orbits of $\psi_d$  for any $d$ divide $\partial M$ into two Birkhoff annuli, $A^d_{out}$ and $A^d_{in}$, where the flow $\psi_d$ points outward from $M$ and inward respectively. The longitude $\lambda$ intersects $g_d$ transversely precisely once for any $d$, pointing in the direction opposite the fiber, i.e. into $A^{d'}_{in}$ and $A^{d}_{in}$. Since $F_n(\lambda) = -\lambda$, we see that $F_n(A^d_{out}) = A^{d'}_{in}$ and $F_n(A^d_{in}) = A^{d'}_{out}$.
Moreover, the normal direction to $\partial M$ is unchanged by the gluing, i.e. the $(v,v^\perp)$ plane in one copy of $M$ is matched with the $(v,v^\perp)$ plane in the other.

We define the flow $\Psi_k$ to be the smooth flow resulting from gluing $\psi_{1+ 6k}$ to $\psi_{1+ 6(2n-k)}$ via $F_n$.


Although the orbit-fiber coordinates are now not a basis for the homology of the boundary torus, as the regular fiber intersects an orbit $g_d$ multiple times when $d>1$, the coordinates $(v,v^\perp,\rho)$ are still a local basis for the unit tangent bundle at each point.

By Remark \ref{rem:fiber}, the fiber $h$ has the form $[h]=-d[\lambda_2]+6[g_{d}]$. It is thus mapped to $[F_n(h)]=d[\lambda_1]-6[g_d']$ for the  corresponding degree $d'$.
In particular, a fiber is not glued to a fiber and the resulting manifold is not globally Seifert fibered but is a graph manifold.

As $d'[\lambda]=6[g_d']-[h]$, the slope of $\lambda$ is $(\frac{6}{d'}, -\frac{1}{d'})$ in the local orbit-fiber directions on the boundary torus for $d'=1+6(2n-k)$.
Thus, $[F_n(h)]=-d[\lambda]-6[g_d']$ is of slope $(\frac{-6d}{d'}-6,\frac{d}{d'})$.
Therefore, the image of the fiber direction $DF(\rho)$ is always tangent to the boundary torus, and is in the 
second and fourth quadrant in the (local) orbit-fiber basis.

This direction falls in the second and fourth quadrant in the $(v^\perp,\rho)$ plane along the core curve of the 
Birkhoff annulus $A_{out}^d$. 
In the 
second Birkhoff annulus $A_{in}^d$, one flows inward to get back to the original copy, and thus the gluing is performed via $F_n^{-1}$.


\vskip .5cm
\begin{claim} The flow $\Psi_k$ is Anosov for any $k$.
\end{claim}

This follows immediately from \cite[Propositions 3 and 4]{HandelThurston}, and we give the idea of the proof here for sake of completeness.
We first show there exist two continuous plane fields, $F^u$ and $F^s$, that intersect along the direction $F^t$ tangent to the flow $\Psi_k$, and are invariant under $D\Psi_k$.

Consider at each point $x$ of $M_n$ the union of the second and fourth quadrants in the $(v^{\perp},h)$ plane, product with the $v$ direction. That is, the portion of the tangent space $T_xM_n$ consisting of the two infinite wedges
\[
\mathcal{W}^u=\big\{(av+bv^\perp,c)\,|\, bc\leq0\big\}.
\]
Within each piece, the action of $D\Psi_k$ on the tangent directions is given by the action of $\psi_d$ and $\psi_{d'}$. Under these actions, both $v^\perp$ and $h$ are moved towards the unstable direction, and thus into the interiors of the second and fourth quadrants.
When passing through the gluing, one uses the action of $DF$ when passing from the back copy to the front one, and the action of $DF^{-1}$ when passing from the front copy to the back one.
In both of these cases, by the computation above the claim, $v^\perp$ is invariant, while $h$ is mapped into the interior of the second quadrant.
It follows that
\[
D\Psi_k^t(\mathcal{W}^u_x)\subset \mathcal{W}^u_{\Psi_k^t(x)}.
\]

By fixing any point $x\in N$ and considering
$D\Psi_k^n(\mathcal{W}^u_{\Psi_k^{-n}(x)})$, these is a sequence of closed sets (each corresponding to a closed set of possible slopes in the $(v^\perp,h)$ plane), that are each contained in all its prequels. Thus, there is an invariant plane field $W^{u}_x=(av+bv^\perp,-\eta(x)b\,)_x\subset T_xN$ at any point $x$ with some finite slope function $\eta(x)>0$. This plane field is the weak stable direction.

The continuity of $\eta$ and $W^u$ follows from the continuity of the $D\Psi_k^t$ action.
By the same argument applied to the action of $\Psi_k^{-t}$ on the first and third quadrants, a continuous invariant plane field $W^{s}$ exists as well.

Next we show there exist continuous $D\Psi_k$ invariant line fields $E^{u}\subset W^u$ and $E^{s}\subset W^s$ which together with $E^t$ yield the Anosov directions for $\Psi_k^t$. We again use a cone field argument.

Any vector in the plane field $W^u$ can be expressed as $(av+bv^\perp,-\eta(x)b\,)_x$. A vector $(bv^\perp,-\eta(x)b\,)_x$ in the intersection of the plane field with the $(v^\perp,\rho)$ plane is shifted along the $v$ direction some bounded distance, as $\rho$ is shifted. As one increases the $v$ component, so $|a|$ becomes larger, the amount of the shift decreases, as both $v$ and $v^\perp$ are invariant under the gluing.

On the other hand, within each piece, the flow takes a vector $(av+bv^\perp,-\eta(x)b\,)_x$ exponentially fast towards the unstable direction $(v^\perp,-1)$ in the intersection of the $(v^\perp,\rho)$ plane and the plane field (as we already know the plane field is invariant).
It follows that for such a vector with $|a|$ large enough, its image under  $D\Psi_k$ is a vector with $|a|$ exponentially decreasing. 
Therefore, we can define the cone field within $W^u$:

\[
\mathcal{C}^u=\big\{\big(a(x)v+bv^\perp,-\eta(x)b\,\big)_x\,|\, b>0\big\}.
\]
Here, $a(x)$ is a function that is a large enough constant $a_0$ on most of $M_n$, decreasing on each piece when approaching $A_{out}^d$, and we can fix this decrease to be slower than the decrease induced by  $D\Psi_0$, but still small enough at the annulus, so that the image under the gluing satisfies $|a|<a_0$. 

It follows that the cone field is invariant under the flow
\[
D\Psi_k^t(\mathcal{C}^u_x)\subset \mathcal{C}^u_{\Psi_k^t(x)},
\]

and there exist invariant line fields $E^u$ and $E^s$.
Thus, the resulting flow is indeed Anosov, as the stable and unstable directions are globally defined.


\vskip .5cm
\begin{claim}
The flows $\Psi_k$ and $\Psi_m$, where $0\leq k,m \leq 2n$ are conjugate only if $m=k$ or $m = 2n-k$.
\end{claim}

Considering the manifold $M_n=M\cup_{F_n} M$, since $F$ does not map a regular fiber in one copy of $M$ to the regular fiber in the second copy, $M_n$ is a graph manifold that is not a Seifert fibered manifold. Thus, its JSJ torus is unique up to isotopy \cite{JacoShalen1979, johannson}. 

Suppose that $H$ were a self-homeomorphism of $M_n$, carrying $\Psi_k$ to $\Psi_m$, demonstrating that they are topologically equivalent flows.  For ease of exposition set $c = 1+ 6m$ and $c' = 1+6(2n-m)$ so that $\Psi_m$ is the result of gluing $\psi_c$ to $\psi_{c'}$.

Let $T\subset N$ be the embedding of $T^1\partial \M$ into $N$. By construction, $T$ is a Birkhoff torus for $\Psi_k$. 
The homeomorphism $H$ takes $T$ to a Birkhoff torus $H(T)$ for $\Psi_m$. At the same time, $T$ is also a Birkhoff torus for $\Psi_m$ and they are isotopic by the uniqueness of the JSJ torus. As the only tangent orbits in the geodesic flow that can be isotoped into the boundary $\partial M$ are $g$ and $g^{-1}$, this is also true for all of its covers. Thus, $T$ and $H(T)$ share the same tangent orbits, and $H(T)$ can be isotoped to $T$ along the flowlines transverse to $T$.

This implies that the two pieces of $M_n \setminus H(T)$, equipped with $\Psi_m$, are orbit equivalent to $(M,\psi_c)\cup(M,\psi_{c'})$. 
Thus, $H$ carries $\{(M,\psi_d),(M, \psi_{d'})\}$ to $\{(M, \psi_c),(M, \psi_{c'})\}$ and it follows from Proposition~\ref{pro:TwoOrbits} that 
$\{d,d'\}=\{c,c'\}$ and therefore either $k=m$ or $k = 2n-m$.
\end{proof}

\begin{remark}
 Handel--Thurston point out in \cite{HandelThurston} that one can take covers of the geodesic flows in the pieces, which for them are denoted by $N_i$:   ``The family we have described can be enlarged by taking the $n$-fold cover of $N_i$ ($n$ independent of $i$) corresponding to the $S^1$ fiber, before gluing the $N_i$ together." One might ask if this could lead to a generalization of our constructions presented here.  However it is not clear which basis to use in the general setting, and how to control the manifold obtained by gluing (the orbit-fiber coordinates do not work for this setting, as the orbit and the fiber necessarily intersect multiple times in the cover).
 \end{remark}
 \begin{remark}
 It is possible to greatly generalize our techniques. Any orbifold or surface with boundary will have a unit tangent bundle that is a suspension with a periodic (or trivial) monodromy, and thus will cover itself in infinitely many different ways, with some covers being orientation reversing. For any hyperbolic surface this yields infinitely many Anosov flows on a manifold with boundary.
To glue such pieces to each other, one must compute the slopes of the periodic orbits on the boundary, the number of orbits for each lift of the geodesic flow, as well as the orientations of the flows through the Birkhoff annuli. In general, this may depend on the choice of a section.
\end{remark}

\bibliographystyle{plain}
\bibliography{Bibliography.bib}

\end{document}